\documentclass{article}
\usepackage[english]{babel}
\usepackage{amsmath,amsfonts,amsthm,amssymb,amscd}
 \usepackage{verbatim}
 \usepackage[usenames]{color}

\renewcommand{\Re}{\mathop{\rm Re}\nolimits}

\newcommand{\p}{\partial}
\newcommand{\e}{\varepsilon}

\newcommand{\ri}{{\rightarrow}}
\newcommand{\C}{{\mathbb C}}
\newcommand{\R}{{\mathbb R}}

\newcommand{\la}{\lambda}

\newcommand{\ty}{\infty}

\newcommand{\om}{\omega}

\newcommand{\DD}{{\cal D}}
\newcommand{\EE}{{\cal E}}
\newcommand{\FF}{{\cal F}}
\newcommand{\GG}{{\cal G}}
\newcommand{\HH}{{\cal H}}

\newcommand{\UU}{{\cal U}}
\newcommand{\VV}{{\cal V}}

\newcommand{\lag}{\langle}
\newcommand{\rag}{\rangle}

\newcommand{\dd}{{\textup d}}

\theoremstyle{plain}
\newtheorem{theorem}{Theorem}[section]
\newtheorem{lemma}[theorem]{Lemma}
\newtheorem{proposition}[theorem]{Proposition}

\newtheorem{condition}[theorem]{Condition}
\theoremstyle{remark}
\newtheorem{remark}[theorem]{Remark}

\newcommand{\de}{\delta}
\numberwithin{equation}{section}
\begin{document}

\date{}
\title{Global exact controllability in
infinite time of    Schr\"odinger equation: multidimensional case}

\date{}

\author{Vahagn Nersesyan, Hayk
Nersisyan}
 \maketitle

{\small\textbf{Abstract.}  We prove that the multidimensional Schr\"odinger equation is exactly controllable in infinite time near any point which is a finite linear combination of eigenfunctions of the Schr\"odinger operator. We prove that, generically with respect to the potential, the linearized system is controllable in infinite time.  Applying the inverse mapping theorem, we prove the controllability of the nonlinear system. }\\

    \tableofcontents

\section{Introduction}\label{S:intr}
This paper is concerned with the problem of  controllability for the following Schr\"odinger equation
\begin{align}
i\dot z &= -\Delta z+V(x)z+ u(t) Q(x)z,\,\,\,\,x\in D, \label{E:hav1}\\
z\arrowvert_{\partial D}&=0,\label{E:ep1}\\
 z(0,x)&=z_0(x),\label{E:sp1}
\end{align}  where   $D\subset\R^d, d\ge1$  is a rectangle, $V,Q \in C^\ty(\overline{D},\R)$ are   given functions,  $u$ is the control,
and $z$ is the state. We prove that  (\ref{E:hav1})-(\ref{E:sp1}) is exactly controllable in infinite time near any point which is  a finite linear combination of eigenfunctions of the Schr\"odinger operator, extending the results of \cite{NN1} to the multidimensional case.
   
   Recall that in the papers \cite{BCH,BeCo, KBCL} it is proved that the 1D Schr\"odinger equation is exactly controllable in finite time in  a neighborhood  of  any  finite linear combination  of eigenfunctions of Laplacian. In 
\cite{CH, PISI, PMMS},      approximate    controllability   in~$L^2$ is proved for multidimensional  Schr\"odinger equation, generically with respect to     functions   $V, Q$ and   domain $D$. In \cite{Mir, BM, VaNe, VN,MM11},    stabilization results and        approximate controllability properties are proved.  In particular,  combination  of the    results   of \cite{VaNe}
with the above mentioned local exact controllability properties   gives   global exact controllability   in finite time for 1D case in the spaces $H^{3+\e}, \e>0$.
  See also  
papers  \cite{RSDR,TR,ALT,ALAL,AC,BCMR}     for controllability of finite-dimensional
systems and   papers
\cite{L,MZ,BP,Z,DGL,PERV} for controllability properties of   various  Schr\"odinger systems.

 The linearization   of   (\ref{E:hav1})-(\ref{E:sp1}) around   the trajectory $e^{-i\la_{k,V}t}e_{k,V}$ with   $u=0$ and $z_0=e_{k,V}$ ($e_{k,V}$ is an eigenfunction of the Schr\"odinger operator $-\Delta +V$ corresponding to some eigenvalue $\la_{k,V}$) is of the form
\begin{align}
i\dot z &= -\Delta z+V(x)z+ u(t) Q(x)e^{-i\la_{k,V}t}e_{k,V},\,\,\,\,x\in D, \label{E:havZ1}\\
z\arrowvert_{\partial D}&=0,\label{E:epZ1}\\
z(0,x)&=0.\label{E:spZ1}
\end{align}  Writing this in the Duhamel form
 \begin{equation}\label{E:lpos}z(T)= -i\int_0^T S(T-s)[ u(s) Qe^{-i\la_{k,V}s}e_{k,V}]\dd s,
\end{equation}  where $S(t)=e^{it(\Delta-V)}$
is the free evolution, we see that (\ref{E:havZ1})-(\ref{E:spZ1}) is equivalent to the following moment problem for $d_{mk}:=\frac{ie^{i\la_{m,V}T}}{\langle
Qe_{m,V },e_{k,V }\rangle}\lag z(T), e_{m,V}\rag$
\begin{equation}\label{E:gczqs}
d_{mk}=  
\int^T_0 e^{i\om_{mk} s} u(s)\dd s,\,m\ge1, \quad\om _{mk}=\la_{m,V}-\la_{k,V}.
\end{equation} It is well known that a gap condition for the frequencies $\om _{mk}$ 
 is necessary  for the solvability of this moment problem when $T<+\ty$ (e.g., see \cite{Y80}). The asymptotic formula for the eigenvalues $\la_{m,V}\sim C_dm^{\frac{2}{d}} $ implies that there is no gap in the case $d\ge3$ (when $d=2$, existence of a domain for which there is a gap between the eigenvalues is an open problem). Moreover, it follows from \cite{A74}    that  there is a linear dependence between the exponentials: there is a non-zero $\{c_m\}\in\ell^2$ such that $\sum_{m=1}^{+\ty}c_me^{i\om_{mk} s} =0$ for $t\in [0,T]$. Hence  (\ref{E:havZ1})-(\ref{E:spZ1}) is non-controllable in finite time $T<+\ty$. The situation is different when $T=+\ty$. Indeed, by Lemma 3.10 in \cite{VN},  the exponentials are independent on  $[0,+\ty)$, and moreover, (\ref{E:havZ1})-(\ref{E:spZ1}) is controllable, by Theorem 2.6 in \cite{NN1}. In \cite{NN1}, we used the controllability of linearized system (\ref{E:havZ1})-(\ref{E:spZ1}) to prove the controllability of nonlinear system only in the case $d=1$. In the multidimensional case, we were able to prove the controllability of (\ref{E:havZ1})-(\ref{E:spZ1})  in a more regular Sobolev space   than the one where nonlinear system (\ref{E:hav1})-(\ref{E:sp1}) is well posed. We do not know if this difficulty of loss of regularity can be treated using  the Nash--Moser  inverse function theorem in the spirit of \cite{BCH}. More precisely, in the multidimensional case, it is very difficult to prove that the inverse of the linearization satisfies the estimates in the   Nash--Moser  theorem. 
   In this paper, we find a space $\HH$ (see (\ref{E:ddsah}) for the definition), where the nonlinear problem is well posed and the linearized problem is controllable.  Applying the inverse inverse function theorem in the space $\HH$, we get controllability for (\ref{E:hav1})-(\ref{E:sp1}). Let us notice that $\HH$ is a sufficiently large space of functions, it contains the Sobolev space $H^{3d}$. Thus, in particular, we prove controllability in $H^{3d}$.
The result of this paper is optimal in the sense that it seems that the multidimensional Schr\"odinger equation  (\ref {E:hav1})-(\ref{E:sp1}) is not exactly controllable in finite time. 

\vspace{10pt}

\textbf{Acknowledgments.}  The authors would like to thank J.-P Puel for providing them in privat communication \cite{JPP}  some results about      regularity questions for the Schr\"odinger equation.

\vspace{22pt} \textbf{Notation}\\\\ In this paper, we use the
following notation. 
Let us define the Banach spaces \begin{align*}
\ell^2&:=\{\{a_j\}\in \C^\ty: \|\{a_j\}\|_{\ell^2}^2=\sum_{j=1}^{+\ty} |a_j|^2<+\ty\}, \\
\ell^2_0&:=\{\{a_j\}\in \ell^2: a_1\in \R\},\\
\ell^\ty&:=\{\{a_j\}\in \C^\ty: \|\{a_j\}\|_{\ell^\ty}=\sup_{j\ge1}  |a_j|<+\ty\},\\
\ell^\ty_0&:=\{\{a_j\}\in \ell^\ty:  \lim_{j\ri+\ty} a_j=0  \},\\
\ell^\ty_{01}&:=\{\{a_j\}\in \ell^\ty_0:  a_1\in \R  \}.
\end{align*}  
  We denote by  $H^s:=H^s(D)$    the Sobolev space of order $s\ge0$.
Consider the Schr\"odinger operator  $ -\Delta +V $, $V \in C^\ty(\overline{D},\R)$  with $ \DD(- \Delta +V):=H_0^1
\cap H^2 $. Let
$\{\la_{j,V} \}$ and $\{e_{j,V} \}$ be the sets of eigenvalues and
normalized eigenfunctions  of this operator.  Let $\langle\cdot,\cdot\rangle$ and~$\|\cdot\|$ be the scalar
product and the norm in the space $L^2 $. Define the space~$H^s_{(V)}:=D((-\Delta+V)^\frac{s}{2})$ endowed with the norm $~\|\cdot\|_{s,V}= \|(\la_{j,V})^\frac{s}{2} \langle\cdot,e_{j,V}\rangle\|_{\ell^2}$.
 When $D$ is the rectangle  $(0,1)^d $      and $V  (x_1,\ldots,x_d) =V_{1} (x_1)+\ldots+ V_{d} (x_d)  $, $V_k   \in C^\ty( [0,1],\R)$,   the eigenvalues and eigenfunctions of $ -\Delta +V $ on $D$ are of the form  
\begin{align}
\la_{j_1,\ldots, j_d,V}&=\la_{j_1,V_1}+\ldots+\la_{j_d,V_d},\label{E:j1}\\
e_{j_1,\ldots, j_d,V}(x_1,\ldots,x_d)&=e_{j_1,V_1}(x_1)\cdot\ldots\cdot e_{j_d,V_d}(x_d),\,\,\,\,\, (x_1,\ldots,x_d)\in D, \label{E:j2}
\end{align}where $\{\la_{j,V_k} \}$ and $\{e_{j,V_k} \}$ are the     eigenvalues and
  eigenfunctions  of operator $-\frac{\dd^2}{\dd x^2}+V_k$  on $(0,1)$.    
 Define the spaces
 \begin{align}\label{E:ddsah}
 \HH&=\{z\in   L^2:( j_1^3\cdot\ldots\cdot j_d^3)  \lag z,  e_{{j_1,\ldots, j_d,V}}   \rag \in \ell_0^\ty,\nonumber\\&\quad\quad\quad\quad \quad\quad  \| z \|_{\HH}:=  \|( j_1^3\cdot\ldots\cdot j_d^3)  \lag z, e_{{j_1,\ldots, j_d,V}}   \rag\|_{\ell^\ty}<+\ty\}, \\
 \VV&=\{z\in L^2:\| z \|_{\VV} :=\!\!\!\!\!\!\sum_{j_1,\ldots,j_d=1}^{+\ty}\!\!\!  (  j_1^3\cdot\ldots\cdot j_d^3 )|\lag z, e_{{j_1,\ldots, j_d,V}}   \rag| <+\ty\} .\label{E:ffaz}
\end{align}   
    The eigenvalues and eigenfunctions of Dirichlet Laplacian on the interval  $(0,1)$ are   $\la_{k,0}=   k^2\pi   ^2$ and $e_{k,0} (x)=\sqrt {2}  \sin  ( {k\pi} x  ) $, $x\in(0,1)$.  It is well known that for any $V\in L^2([0,1],\R)$
 \begin{align}
& \la_{k,V}=k^2\pi^2+\int_0^1 V(x)\dd x+r_k,\label{E:app1}\\
 & \|e_{k,V}-e_{k,0}\|_{L^\ty}\le \frac{C}{k},\label{E:app2}\\
 &  \Big\|\frac{\dd e_{k,V}}{\dd x}-\frac{\dd e_{k,0}}{\dd x} \Big\|_{L^\ty}\le C,\label{E:app3}
 \end{align} where $\sum_{k=1}^{+\ty}r_k^2<+\ty$ (e.g., see \cite{PT}).
   For a Banach space $X$, we shall denote by $B_X(a, r)$ the open ball of radius $r > 0$ centered at  $a\in X$. The integer part of $x\in\R$ is denoted by $[x]$.   We denote by $C$ a constant whose value may change from line to line.

\section{Main results}\label{S:NL}

\subsection{Well-posedness of Schr\"odinger equation}\label{S:APROX1}

We assume that
$V  (x_1,\ldots,x_d) =V_{1} (x_1)+\ldots+ V_{d} (x_d), x_k\in[0,1] $ and $V_k   \in C^\ty( [0,1],\R) , k=1,\ldots,d $.
Let us   consider the following   Schr\"odinger equation
\begin{align}
i\dot z &= -\Delta z+V(x)z+ u(t) Q(x)z+v(t)Q(x)y,\,\,\,\, \label{E:hav5}\\
z\arrowvert_{\partial D}&=0,\label{E:ep5}\\
 z(0,x)&=z_0(x).\label{E:sp5}
\end{align}  
The following lemma shows the well-posedness of this system in $H_{(V)}^2$.
\begin{lemma}\label{L:LD}

For any   $z_0\in H_{(V)}^2 $,      
  $u,v\in L^1_{loc}([0,\ty),\R)$ and $y\in L^1([0,\infty), H_{(V)}^2)$
  problem
(\ref{E:hav5})-(\ref{E:sp5}) has a unique solution   $ z\in C([0,\infty), H_{(V)}^2) $. 
Furthermore, if $v=0$, then    for all $t\ge0$ we have
\begin{align}
\|z(t)\| =\|z_0\|.    \label{E:barev} 
\end{align}
\end{lemma}
See  \cite{CW} for the proof.  In \cite{KBCL} it is proved that this problem is well posed in $H^3_{(V)}$ for $d=1$, and in \cite{JPP} the well-posedness in $H^3_{(V)}$  is proved for $d\ge1$. 

 For any integer  $l\ge3$, let $m=m(l):=[\frac{l-1}{2}]$ and define the space 
$$
C^m_0:=\{u\in C^m([0,\ty),\R): \frac{\dd ^ku}{\dd t^k} (0)=0,k\in[0,m]\}
$$
endowed with the norm of $C^m([0,\ty),\R)$.
The following lemma shows that   
problem (\ref{E:hav5})-(\ref{E:sp5}) is well posed in  higher Sobolev spaces when $u,v$ and $y$ are more regular.

\begin{lemma}\label{L:LmpD}

For any integer $l\ge3$,   any  $z_0\in H^l_{(V)}$, any  $y\in W^{m,1}_{loc}([0,\ty),H^2_{(V)})$ and any  $u,v\in C^{m}_{0} $   the solution $z$ in Lemma \ref{L:LD}
 belongs to the space  $ 
C([0,\infty),H^l)\cap C^1([0,\infty),H^{l-2})$. Moreover,    there is a constant $C>0$ such that  
\begin{align}\label{E:gnar}
\|z(t)\|_{H^l}+\|z\|_{W^{m,1} ([0,t],H^2_{(V)})}\le &C(\|z_0\|_{l,V}+\|v\|_{C^m_0 }\|y\|_{W^{m,1} ([0,t],H^2_{(V)})}) \nonumber\\&\times e^{C(\|u\|_{C^m_0 }+1)t}.
\end{align}
 \end{lemma}
See  Appendix of \cite{BCH} for the proof. 
\begin{lemma}\label{L:oe}
Denote by
$\UU_t(\cdot,\cdot): H^{2}_{(V)}\times L^1_{loc}(\R_+,\R)  \ri  H^{2}_{(V)}$ the resolving operator of
(\ref{E:hav1}), (\ref{E:ep1}). Then $\UU_t(\cdot,\cdot)$ is locally Lipschitz continuous:    
  there is $C>0$ such that
\begin{align}\label{E:ppz}
 \|\UU_t(z_0,u)-\UU_t(z_0',u')\|_{H^l}\le   C(\|z_0-z_0'\|_{l, V}+\|u-u'\|_{C^m_0 } \|z_0'\|_{l,V} ) e^{C(\|u\|_{C^m_0 }+1)t}.
\end{align} \end{lemma}
\begin{proof}
Notice that $z(t):=\UU_t(z_0,u)-\UU_t(z_0',u')$ is a solution of problem \begin{align*}
i\dot z &= -\Delta z+V(x)z+ u(t) Q(x)z+(u(t)-u'(t))Q(x) \UU_t(z_0',u'),  \\
z\arrowvert_{\partial D}&=0, \\
 z(0,x)&=z_0(x)-z_0'(x).  
\end{align*}  Applying Lemma  \ref{L:LmpD}, we get 
\begin{align}
&\|z(t)\|_{H^l} \le C(\|z_0-z_0'\|_{l,V}+\|u-u'\|_{C^m_0 }\|\UU_\cdot(z_0',u')\|_{W^{m,1} ([0,t],H^2_{(V)})}) e^{C(\|u\|_{C^m_0 }+1)t},\label{E:gndar}
\\ &\|\UU_\cdot(z_0',u')\|_{W^{m,1} ([0,t],H^2_{(V)})}\le  C \|z_0'\|_{l,V}  e^{C(\|u\|_{C^m_0 }+1)t}.\label{E:gndar2}
\end{align}
Replacing (\ref{E:gndar2}) into (\ref{E:gndar}), we get (\ref{E:ppz}).
\end{proof}
  Let us rewrite (\ref{E:hav1})-(\ref{E:sp1})  in the Duhamel form
\begin{equation}\label{E:DU}
z(t)=S(t) z_0-i\int_0^tS(t-s)[u(s)Qz(s)]\dd s,
\end{equation}
where $S(t)=e^{it(\Delta-V)}$
is the free evolution.
 Let us take any $w\in L^1(\R_+,\R)$ and estimate the following integral
$$
G_t(z):=\int_0^tS(-s) [w(s)Qz(s)]\dd s.
$$  We take controls from the weighted space space 
$$
\GG:=\{u\in L^1(\R_+,\R): u(\cdot)e^{B\cdot}\in L^1(\R_+,\R) \}
$$endowed with the norm $\|u\|_\GG=\|u(\cdot)e^{B\cdot}\|_{L^1}$, where   the constant $B>0$ will be chosen later.  For $B>C+1$, where $C$ is the constant in Lemma \ref{L:LmpD}, we have the following result.
 \begin{proposition}\label{L:lav}Let us take  any     $l\ge {4d}$,
     $z_0\in H^l_{(V)}, w\in\GG$ and $u \in    C^m_0 $, and let $z(t):=\UU_t(z_0,u).$  Then there are  constants $\de, C>0$ such that for any $u\in B_{ C^m_0}(0,\de) $ and for any $t>s\ge0$
        \begin{align}
\| G_t(z)- G_s(z)\|_\HH\le C \int_s^t \|z(\tau)\|_{H^l} |w(\tau)|\dd \tau,  \label{E:S12}
\end{align} 
     and  the following integral converges in $\HH$
    \begin{align}
G_\ty(z):=\int_0^{+\ty}S(-\tau)[w(\tau)Qz(\tau)]\dd \tau.
  \label{E:S1}
\end{align}   \end{proposition}
\begin{proof}

 Using    (\ref{E:gnar}) with $v=0$,   the definition of $\GG$, and choosing $\de>0$ sufficiently small, we see that   
 $$
 \int_0^{+\ty} \|z(\tau)\|_{H^l} |w(\tau)|\dd \tau<+\ty.
 $$Combining this with (\ref{E:S12}), we prove the convergence of the integral in (\ref{E:S1}).
 Let us prove (\ref{E:S12}).  To simplify the notation, let us suppose that $d = 2$; the proof of the general case is similar.
Let $V (x_1, x_2) = V_1(x_1)+ V_2(x_2)$.
Integration by parts   gives  
\begin{align*}
\lag Q z(s),   e_{j_1,V_1} e_{j_2,V_2}    \rag=&\frac{1}{\la_{j_1,V_1}  } \lag (-\frac{\p^2}{\p x^2_1}+V_1)(Q z),  e_{j_1,V_1} e_{j_2,V_2}   \rag      \nonumber\\=&\frac{1}{\la_{j_1,V_1}^2 } \lag (-\frac{\p^2}{\p x^2_1}+V_1)(Q z),  (-\frac{\p^2}{\p x^2_1}+V_1)e_{j_1,V_1} e_{j_2,V_2}     \rag \nonumber \\ =&\frac{1}{\la_{j_1,V_1 }  ^2 }  \int_0^1    \frac{\p^2}{\p x^2_1} (Q z)    e_{j_2,V_2}  \dd x_2\frac{\p}{\p x_1} e_{j_1,V_1}   \big|_{x_1=0}^{x_1=1}\nonumber\\&+ \frac{1}{\la_{j_1,V_1}^2 }\Big(\lag V_1(-\frac{\p^2}{\p x^2_1}+V_1)(Q z),  e_{j_1,V_1} e_{j_2,V_2}     \rag\nonumber\\&+\lag \frac{\p}{\p x_1}(-\frac{\p^2}{\p x^2_1}+V_1)(Q z),    \frac{\p}{\p x_1} e_{j_1,V_1} e_{j_2,V_2}   \rag\Big) \nonumber \\ =&:I_{j}+J_j.
\end{align*} Let us estimate $I_j$. Since  
$ \frac{\p^2}{\p x_1^2}  (Q z(s))=0 $ for all $x_1\in [0, 1]$ and for $x_2=0$ and $x_2 =1$, integration by parts in $x_2$ implies
\begin{align}
I_j& =\frac{1}{\la_{j_1,V_1 }  ^2\la_{j_2,V_2 }   }  \int_0^1  (-\frac{\p^2}{\p x^2_2}+V_2) \Big( \frac{\p^2}{\p x^2_1} (Q z)\Big)  e_{j_2,V_2}    \dd x_2  \frac{\p}{\p x_1}   e_{j_1,V_1}\big|_{x_1=0}^{x_1=1}\nonumber\\&=\frac{1}{\la_{j_1,V_1 }  ^2\la_{j_2,V_2 } ^2  }  \int_0^1  (-\frac{\p^2}{\p x^2_2}+V_2) \Big( \frac{\p^2}{\p x^2_1} (Q z)\Big)  (-\frac{\p^2}{\p x^2_2}+V_2) e_{j_2,V_2}    \dd x_2  \frac{\p}{\p x_1}   e_{j_1,V_1}\big|_{x_1=0}^{x_1=1}\nonumber\\&=\frac{1}{\la_{j_1,V_1 }  ^2\la_{j_2,V_2 } ^2  }   (-\frac{\p^2}{\p x^2_2}+V_2) \Big( \frac{\p^2}{\p x^2_1} (Q z)\Big)  \frac{\p}{\p x_2} e_{j_2,V_2}       \frac{\p}{\p x_1}   e_{j_1,V_1}\big|_{x_2=0}^{x_2=1}\big|_{x_1=0}^{x_1=1}\nonumber\\&\quad + \frac{1}{\la_{j_1,V_1 }  ^2\la_{j_2,V_2 } ^2  }  \int_0^1 V_2 (-\frac{\p^2}{\p x^2_2}+V_2) \Big( \frac{\p^2}{\p x^2_1} (Q z)\Big)    e_{j_2,V_2}    \dd x_2  \frac{\p}{\p x_1}   e_{j_1,V_1}\big|_{x_1=0}^{x_1=1}\nonumber\\&\quad + \frac{1}{\la_{j_1,V_1 }  ^2\la_{j_2,V_2 } ^2  }  \int_0^1 \frac{\p}{\p x_2}(-\frac{\p^2}{\p x^2_2}+V_2) \Big( \frac{\p^2}{\p x^2_1} (Q z)\Big)    \frac{\p}{\p x_2}  e_{j_2,V_2}    \dd x_2  \frac{\p}{\p x_1} e_{j_1,V_1}\big|_{x_1=0}^{x_1=1}\nonumber\\ &=:I_{j,1}+I_{j,2}+I_{j,3}.\end{align}
Let us consider the term $I_{j,1}$:
\begin{align*} 
&I_{j,1}=\Big(\frac{2j_1j_2\pi^2}{\la_{j_1,V_1 }  ^2\la_{j_2,V_2 } ^2  }   (-\frac{\p^2}{\p x^2_2}+V_2) \Big( \frac{\p^2}{\p x^2_1} (Q z)\Big)   \cos(j_1\pi x_1) \cos(j_2\pi x_2)      \nonumber\\& \!+\!\frac{1}{\la_{j_1,V_1 }  ^2\la_{j_2,V_2 } ^2  }  \! (-\frac{\p^2}{\p x^2_2}\!+\!V_2) \Big( \frac{\p^2}{\p x^2_1} \!(Q z)\Big)\!  \frac{\p^2}{\p x_1\p x_2}(e_{j_1,V_1}\! e_{j_2,V_2}\!-\!e_{j_1,0} e_{j_2,0}) \!\Big)\big|_{x_2=0}^{x_2=1}\big|_{x_1=0}^{x_1=1}.
\end{align*}
Using (\ref{E:app1}), (\ref{E:app3}) and the Sobolev embedding $H^s\hookrightarrow L^\ty, s>\frac{d}{2}$, we get
\begin{align*}
 \sup_{j_1,j_2\ge1}\Big|j_1^3j_2^3\int_s^te^{i(\la_{j_1,V_1 }+\la_{j_2,V_2 })\tau} w(\tau)I_{j,1}  \dd \tau \Big|\le C \int_s^{t}  \|  z(\tau) \|_{H^l}  |w(\tau) | \dd \tau.     
\end{align*}The Riemann--Lebesgue theorem and (\ref{E:app3}) imply  that
$$
j_1^3j_2^3\int_s^te^{i(\la_{j_1,V_1 }+\la_{j_2,V_2 })\tau} w(\tau)I_{j,1}  \dd \tau\ri0 \quad \text{as $j_1+j_2\ri+\ty.$}
$$
Thus
$$
j_1^3j_2^3\int_s^te^{i(\la_{j_1,V_1 }+\la_{j_2,V_2 })\tau} w(\tau)I_{j,1}  \dd \tau\in \ell_0^\ty.
$$The terms $I_{j,2},I_{j,3}$ and $J_{j}$ are treated exactly in the same way. We omit the details. 
Thus we get that
$$\|G_t(z)-G_s(z)\|_\HH=\|\int_s^tS(-\tau) [w(\tau)Qz(\tau)]\dd \tau\|_\HH
\le C \int_s^{t}  \|  z(\tau) \|_{H^l}  |w(\tau) | \dd \tau.
$$

\end{proof}

Let $T_n\ri+\ty$ be a sequence such that $e^{-i\la_{V,j}T_n}\ri 1$  as $n\ri\ty$ for any $j\ge1$ (e.g., see Lemma 2.1 in \cite{NN1}).
Then 
\begin{align}\label{E:llav}
 S(T_n)    z  \ri   z   \text{ as $n\ri+\ty$ in $ \HH$ for any $z \in \HH$ and $t\ge0$.} \end{align}Indeed, 
since 
\begin{equation}\label{E:azh}S(t)   z   = \sum_{j=1}^{+\ty}e^{-i\la_{j,V}t} \lag   z,e_{j,V}\rag e_{j,V},\end{equation}  
we have
 \begin{align*}\|S(T_n)    z   -  z \|_\HH\le &\sup_{\la_{j_1,\ldots,j_d,V}\le N}   (j_1^3\cdot\ldots\cdot j_d^3)|e^{-i\la_{j_1,\ldots,j_d,V}T_n}-1 ||  \lag   z,e_{j_1,\ldots,j_d,V}\rag| \nonumber\\&+2 \sup_{\la_{j_1,\ldots,j_d,V}> N}   (j_1^3\cdot\ldots\cdot j_d^3) |\lag   z,e_{j_1,\ldots,j_d,V}\rag|  \le \frac{\e}{2}+ \frac{\e}{2} =\e \nonumber
\end{align*} for sufficiently large integers $N,n\ge1$.  

Let us take $t=T_n$ in (\ref{E:DU}) and pass to the limit $n\ri\ty$.  Using Proposition~\ref{L:lav}, the embedding $H^l_{(V)} \hookrightarrow \HH$ and (\ref{E:llav}), we obtain the following result. 
\begin{lemma}Let us take  any     $l\ge {4d}$ and 
     $z_0\in H^l_{(V)}$. There is a constant $\de>0$ such that    for 	any $u\in B_{ C^m_0}(0,\de)\cap \GG $ the following limit exists in $\HH$
\begin{align}\label{E:ldl}
\lim_{n\ri+\ty}\UU_{T_n} (z_0,u)=:\UU_{\ty} (z_0,u).
\end{align}
\end{lemma}

 \subsection{Exact controllability in infinite time
}\label{S:APROX2} 
Let $l\ge4d$ be the integer in Proposition \ref{L:lav}.  Take any integer
 $s\ge l$ and let
$$
H^s_0(\R_+,\R):=\{u\in H^s (\R_+,\R): u^{(k)}(0)=0, k=0,  \ldots, s-1\}.
$$
  The set of admissible controls is the Banach space
\begin{align}\label{E:tty}
\FF:=   \GG     \cap H^s_0(\R_+,\R)
\end{align}
endowed with the norm
$\|u\|_\FF:= \|u\|_{\GG}+ \|u\|_{H^s}$. 
 Equality (\ref{E:barev}) implies that
it suffices to consider the controllability properties of
(\ref{E:hav1}), (\ref{E:ep1}) on the unit sphere $S$ in $L^2$.

 We prove the controllability of (\ref{E:hav1}),  (\ref{E:ep1}) under   below
condition. 

\begin{condition}\label{C:p1}Suppose that  
the   functions  $V,Q \in C^\ty(\overline{D},\R)$ are such that
\begin{enumerate}
\item [(i)] $\inf_{p_1, j_1,\ldots,p_d,j_d\ge1}|{{(p_1j_1 \cdot\ldots\cdot p_dj_d)^3}}   Q_{pj}|\!>\!0,   \! Q_{pj}\!:=\!  \langle Qe_{p_1,\ldots,p_d,V },e_{j_1,\ldots,j_d,V }\rangle $,
   \item [(ii)] $\la_{i,V }-\la_{j,V }\neq \la_{p,V }-\la_{q ,V}$ for
all $i,j,p,q\ge1$ such that
  $\{i,j\}\neq\{p,q\}$ and $i\neq j$.

\end{enumerate}
\end{condition}
See \cite{NN1} and \cite{PISI, VaNe,MS10}  for the proof of genericity of (i) and (ii), respectively. Let us set 
\begin{align}\label{E:etar}
\EE :=\textup{span}\{e_{j,V}\}.
\end{align}

Below theorem  is  the main result  of   this
paper.
\begin{theorem}\label{T:AnvL}
Under Condition \ref{C:p1}, for any $\tilde z\in S\cap \EE  $
 there is  
$\sigma>0$ such that problem (\ref{E:hav1}), (\ref{E:ep1}) is
exactly controllable in infinite time in    $ S\cap
B_{\HH}(\tilde z,\sigma)$, i.e., for any $z_1\in S\cap
B_{\HH}(\tilde z,\sigma)$ there is a control $u\in \FF$ such that limit (\ref{E:ldl}) exists in $\HH$ and  $z_1= \UU_\ty(\tilde z,u).$
\end{theorem} 
See Section \ref{S:1} for the proof. Since the space $H^{3d}_{(V)}$ is  continuously embedded into $\HH$, we obtain
\begin{theorem}\label{T:AnvL}
Under Condition \ref{C:p1}, for any $\tilde z\in S\cap \EE  $
 there is  
$\sigma>0$ such that   for any $z_1\in S\cap
B_{H^{3d}_{(V)}}(\tilde z,\sigma)$ there is a control $u\in \FF$ such that limit (\ref{E:ldl}) exists in $\HH$ and  $z_1= \UU_\ty(\tilde z,u).$

\end{theorem} 
\begin{remark}       As in the case $d=1$ (see Theorems 3.7 and 3.8 in \cite{NN1}) here also one can prove controllability in higher Sobolev spaces with more regular controls, and a global controllability property  using a compactness argument. 
\end{remark}

 \section{Proof of Theorem \ref{T:AnvL}}\label{S:MR}

\subsection{Controllability of   linearized
system}\label{S:Gcayin}

In this section,  we   study  the controllability of the   linearization of   (\ref{E:hav1}), (\ref{E:ep1}) around the trajectory $\UU_t(\tilde z ,0), \tilde z\in S\cap \EE  
$:
\begin{align}
i\dot z &= -\Delta z+V(x)z  + u(t) Q(x)\UU_t(\tilde z,0),\,\,\,\,
\label{E:hav2}\\
z\arrowvert_{\partial D}&=0,\label{E:ep2}\\
 z(0,x)&=z_0.\label{E:sp2}
\end{align}  
   The controllability in infinite time of this system  is proved in \cite{NN1}, Section 2. For the proof of Theorem \ref{T:AnvL} we need to show   controllability of (\ref{E:hav2})-(\ref{E:sp2}) in $\HH$ which is   larger than the space considered in \cite{NN1}.  Hence a generalization  of the arguments of \cite{NN1} is needed.

   Let $S$ be the unit sphere in $L^2$. For  $y\in S$, let $T_y$ be the   tangent space to $S$ at $y\in S$:
$$
T_{y}=\{z\in L^2: \Re\lag z, y\rag=0\}.
$$
 By  Lemma \ref{L:LD}, for any $z_0\in  H_{(V)}^2$ and $u\in L_{loc}^1(\R_+,\R)$,
  problem
(\ref{E:hav2})-(\ref{E:sp2}) has a unique solution   $ z\in
C(\R_+, H_{(V)}^2) $.
 Let 
\begin{align*}
R_t(\cdot, \cdot):
   H_{(V)}^2\times L^1([0,t],\R)&\rightarrow
 H_{(V)}^2,\\ (z_0,u)&\rightarrow z(t)
 \end{align*} be   the resolving operator. 
Then $R_t(z_0,u)\in T_{\UU_t(\tilde z,0)}$ for any $z_0\in T_{\tilde z}\cap H_{(V)}^2$ and
$t\ge0$. Indeed, 
  \begin{align*}
  \frac{\dd}{\dd t}\Re \lag R_t , \UU_t \rag&= \Re \lag \dot R_t , \UU_t \rag+ \Re \lag R_t , \dot \UU_t \rag\\&= \Re \lag i(\Delta-V) R_t  -i u(t) Q(x)\UU_t 
 , \UU_t \rag + \Re \lag R_t , i(\Delta-V) \UU_t \rag \nonumber\\
 &=\Re \lag i(\Delta-V) R_t   
 , \UU_t \rag + \Re \lag R_t , i(\Delta-V) \UU_t \rag=0.
 \end{align*} Since $ \Re \lag R_0 , \UU_0 \rag=  \Re \lag z_0 , \tilde z \rag=0$, we get $R_t(z_0,u)\in T_{\UU_t(\tilde z,0)}$.

 As (\ref{E:hav2})-(\ref{E:sp2}) is a linear control problem, the controllability of system with $z_0=0$ is equivalent to that   with any $z_0\in T_{\tilde z}$. Henceforth, we   take $z_0=0$ in   (\ref{E:sp2}).
Let us rewrite this problem in the Duhamel form
\begin{equation}\label{E:lpo}z(t)= -i\int_0^t S(t-s) u(s) Q(x)\UU_s(\tilde z,0)\dd s.
\end{equation}  
 Let $T_n\ri\ty$ be the sequence defined in Section \ref{S:APROX1}. For any $u\in \FF $ the following limit exists in $\HH$   
\begin{equation}\label{E:eez}R_{\ty}(0,u):=\lim_{n\rightarrow +\ty} z(T_n)=\lim_{n\rightarrow
+\ty}
R_{T_n}(0,u).
\end{equation} 
Using (\ref{E:azh}) and (\ref{E:lpo}), we obtain
\begin{equation}\label{E:gcz}
\lag z(t), e_{m,V}\rag=-i \sum_ {k=1} ^{+\ty} e^{-i\la_{m,V}t}\lag\tilde z,e_{k,V}\rag Q_{mk} 
\int^t_0 e^{i\om_{mk} s} u(s)\dd s,\,m\ge1,
\end{equation}where $\om _{mk}=\la_{m,V}-\la_{k,V}$ and  $Q_{mk}:= \langle
Qe_{m,V },e_{k,V }\rangle $.   Let us take  $t=T_n$
in (\ref{E:gcz}) and pass to the limit as $n \rightarrow +\ty$.   The choice of the sequence $T_n$ implies that
\begin{equation}\label{E:gcza}
\lag R_{\ty}(0,u), e_{m,V}\rag=-i \sum_ {k=1} ^{+\ty} \lag\tilde z,e_{k,V}\rag Q_{mk} 
\int^{+\ty}_0 e^{i\om_{mk} s} u(s)\dd s.
\end{equation}
 Moreover, $R_{\ty}(0,u) \in T_{\tilde z} $. Indeed, using (\ref{E:eez}) and  the convergence $  \UU_{T_n}(\tilde z ,0)\ri \tilde z $ in $ \HH$, we get 

$$
\Re\lag R_{\ty}(0,u),\tilde z\rag=\lim_{n\ri\ty} \Re\lag R_{T_n}(0,u),\UU_{T_n}(\tilde z,0)\rag=0.
$$ 

\begin{lemma}\label{L:gc2}
The mapping  $R_{\ty}(0,\cdot)$ is linear continuous   from $\FF$ to $
T_{\tilde z}\cap \HH$.\end{lemma}
\begin{proof}   

 By (2.24) in \cite{NN1},   there is a constant $C>0$ such that for any $m_j,k_j\ge 1$, $j=1,\ldots,d$ we have
\begin{align}\label{E:reza}
\Big|\frac{(m_1\cdot\ldots\cdot m_d)^3}{(k_1\cdot\ldots\cdot k_d)^3} \langle Qe_{k_1,\ldots,k_d,V },e_{m_1,\ldots,m_d,V }\rangle\Big|\le C.
\end{align}
Then (\ref{E:gcza}), (\ref{E:reza}) and the Schwarz inequality imply that
\begin{align*} 
&\|   R_{\ty}(0,u)\|_\HH =\sup_{m_1,\ldots,m_d\ge1} |  (m_1^3\cdot\ldots\cdot m_d^3) \lag R_{\ty}(0,u), e_{{m_1,\ldots, m_d,V}}   \rag| \\&\le C\!\!\!\sup_{m_1,\ldots,m_d\ge1}\!\! \Big|\! (m_1^3\cdot\ldots\cdot m_d^3) \lag \tilde z, e_{{m_1,\ldots, m_d,V}}   \rag\langle Qe_{m,V },e_{m,V }\!\rangle \!\!\int^{+\ty}_0\!\!\!\!  u(s)\dd s\Big| \!\\&\quad+ C \!\| \tilde z\|_\VV   \!\!\!\sup_{m,k\ge1, m\neq k } \!\! \Big|\!\frac{(m_1\cdot\ldots\cdot m_d)^3}{(k_1\cdot\ldots\cdot k_d)^3}\! \langle Qe_{k_1,\ldots,k_d,V },e_{m_1,\ldots,m_d,V }\!\rangle \!\!\int^{+\ty}_0\!\!\!\! e^{i\om_{mk} s} u(s)\dd s\Big| \\&\le C\| \tilde z\|_\VV ^2 \|u\|_\FF^2<+\ty,
\end{align*} where $\VV$ is defined by (\ref{E:ffaz}).

\end{proof}

Let  us introduce the set   \begin{align*}
\EE_0\!:=\!\{z\in S\!:\!\exists p,q\ge 1,p\neq q,& z=c_pe_{p,V}+c_qe_{q,V}, \nonumber\\& |c_p |^2 \lag Q e_{p,V},e_{p,V}\rag\!-\!| c_q |^2 \lag Q e_{q,V},e_{q,V}\rag = 0  \}.\end{align*}   
\begin{theorem}\label{T:Lin} Under Condition  \ref{C:p1}, for any $\tilde z\in S \cap \EE \setminus \EE_0$, the mapping
$R_{\ty}(0,\cdot):\FF\rightarrow
T_{\tilde z}\cap \HH$   admits a continuous
right inverse, where the space $T_{\tilde z}\cap \HH$ is endowed with the
norm of $\HH$. If $\tilde z\in S \cap   \EE_0$, then $R_{\ty}(0,\cdot)$ is not invertible.
\end{theorem}
\begin{remark}
The invertibility of the mapping $R_{T}(0,\cdot)$  with finite $T>0$ and $\tilde z=e_1$ is studied by Beauchard et al. \cite{BCKY}. They prove that for   space dimension $d\ge 3$ the  mapping  is not invertible.  By Beauchard \cite{BCH}, $R_{T}$ is invertible in the case $d=1$ and $\tilde z=e_1$. The case $d=2$ is open to our knowledge.\end{remark}
 
For any $u\in L^1(\R_+,\R)$,  denote by $\check u$ the inverse Fourier transform of the function obtained by extending $u$ as zero to $\R_-^*$:
\begin{equation}\label{E:eaz}
\check u(\om):= \int^{+\ty}_0 e^{i\om s} u(s)\dd s.
\end{equation}

\begin{proof}[Proof of Theorem \ref{T:Lin}]

  Let us take any $\tilde z\in S \cap \EE \setminus \EE_0$ and $y\in  T_{\tilde z}\cap \HH$. There is an integer $N\ge1$ such that $\lag \tilde z, e_{k,V}\rag=0$ for any $k\ge N+1$. Let us define \begin{equation*} d_{mk}:=\frac{i \lag y, e_{m,V}\rag \lag e_{k,V},\tilde z\rag -i\lag  e_{k,V},y\rag \lag \tilde z,e_{m,V}\rag}{ Q_{mk}  
}+C_{mk},
\end{equation*} for $k\le N$, where $C_{mk}\in\C$. Notice that
 \begin{align*}
\sup_{m,k\ge1 } \Big| \frac{  \lag y, e_{m,V}\rag \lag e_{k,V},\tilde z\rag  }{ Q_{mk}  
} \Big |  &\le   C \| y\|_{\HH} \|\tilde z\|_{\HH} <+\ty.
\end{align*}Repeating the arguments of the proof of  Theorem 2.6 in \cite{NN1}, one can show that  the constants $C_{mk}$ can be chosen  such that  

\begin{align*}
  \sup_{m,k\ge1}    |d_{mk}|  <+\ty, d_{mm}=d_{11},    &  \text{  $d_{mk}= {\overline d}_{km}$ for all $ 1\le m,k\le N,  $} \\&  \text{$d_{mk}\ri 0$ as $m\ri \ty$ for any fixed $k\ge1$},
\end{align*}
and  $y=R_{\ty}(0,u)$ holds for any solution $u\in \FF$ of system 
$$d_{mk} = \check u (\om_{mk}) \quad\text{for all $m\ge1$ and $k\in[1,N]$}. $$  It remains to use the following  proposition, which is proved in next subsection. 
\begin{proposition}\label{P:P1} If the strictly increasing sequence $\omega_m\in \R, m\ge1$ is such that $\om_1=0$ and $\om_m\ri+\ty$ as  $m\ri+\ty$, then
there is a linear continuous operator $A$ from $ \ell^\ty_{01}$ to $  \FF$ such that $\{ \check { A(d)} (\om_{m})\}=d$ for any $d\in  \ell^\ty_{01} $. 
\end{proposition}
 The proof of the non-invertibility of $  R_\ty(0,\cdot)$   is a remark by  Beauchard and Coron \cite{BeCo} (cf. Step 2 of the proof of Theorem 2.6 in \cite{NN1}).

 \end{proof}
 
 \begin{remark}
  The proof of Theorem \ref{T:Lin} does not work in the multidimensional case for a general $\tilde z\notin\EE$. Indeed,  assume that $\lag z, e_{k_n,V}\rag\neq0$ for some sequence $k_n\ri+\ty.$ Then
  the well-known  asymptotic formula for  eigenvalues 
$
\la_{k,V}~\sim~ C_dk^{\frac{2}{d}}
$
  implies that the frequencies $\omega_{m_nk_n}\ri0$ for some integers $m_n\ge1$   for space dimension $d\ge 3$. Thus the moment problem $\check u (\om_{mk})=d_{mk}$ cannot be solved in the space $L^1(\R_+,\R)$ for a general  $d_{mk}$. Clearly, this does not imply the non-controllability in infinite time of linearized system.
 \end{remark}

   \subsection{Proof of Proposition \ref{P:P1}} 
  The proof of Proposition \ref{P:P1} is close to that of Proposition 2.9 in \cite{NN1}.  Let 
   $$
\tilde \GG:=\{u\in L^1(\R_+,\R): u^2(\cdot)e^{\tilde B\cdot}\in L^1(\R_+,\R) \}
$$endowed with the norm $\|u\|_{\tilde\GG}=\|u^2(\cdot)e^{\tilde B\cdot}\|_{L^1}$, where   the constant $\tilde B>2B$.   Then  $$
\tilde \FF:=   \tilde \GG     \cap H^s_0(\R_+,\R)
$$is a subspace of $\FF$ defined by (\ref{E:tty}). Moreover, $\tilde \FF$ is a Hilbert space.
  The construction of the operator $A$ is based on the following lemma.
  \begin{lemma}\label{L:Min} Under the conditions of Proposition \ref{P:P1},  
    for any $d\in \ell^\ty_{01}$   there is $u\in     \tilde\FF  $ such that $\{\check u (\om_{m})\}=d$.
  \end{lemma}
  \begin{proof}[Proof of Proposition \ref{P:P1}]
  By Lemma \ref{L:Min}, the   mapping $ u\ri\{\check u (\om_{m})\}
$ 
 is surjective linear bounded form Hilbert space $\tilde\FF$ onto Banach space $\ell^\ty_{01}$. Hence it admits a linear bounded right inverse $A :\ell^\ty_{01}\ri \tilde \FF$. 
\end{proof} 
\begin{proof}[Proof of Lemma \ref{L:Min}]
Let us show that  there is a constant $M>~0$ such that 
  for any $d\in \ell^\ty_{01}, \|d\|_{\ell^\ty_{01}} \le1$   there is $u\in   {B_{ \tilde\FF}(0,M)}$ satisfying $\{\check u (\om_{m})\}=~d$.

 Let us   introduce the functional
$$
F(u):=\| \{\check u (\om_{m})\}-d\|_{ \ell^\ty}    
$$
defined on the space $\tilde \FF$.

\vspace{6pt}\textbf{Step 1.} First, let us show that for any $M>0$   there is $u_0\in \overline {B_{\tilde \FF}(0,M)} $ such  that
\begin{equation}\label{E:mh}
F(u_0)= \inf_{u\in \overline {B_{ \tilde\FF}(0,M)} } F(u).
\end{equation}To this end, let $u_n\in \overline {B_{\tilde \FF}(0,M)}  $ be an arbitrary  minimizing sequence. Since $  \tilde \FF$ is reflexive, without loss of generality, we can assume that  there is $u_0\in \overline {B_{  \tilde \FF}(0,M)} $ such that $u_n \rightharpoonup u_0 $  in $  \tilde \FF$. Using the compactness of the injection $H^1([0,N])\ri C([0,N])$ for any $N>0$ and a diagonal extraction, we can assume that $u_n  (t)\ri u_0 (t)$ uniformly for $t\in[0,N]$.   Again extracting a subsequence, if it is necessary, one gets $\{\check u_n(\om_{m})\} \ri  \{\check u_0(\om_{m})\} $  in $  \ell ^\ty$ as $n\ri+\ty.$ Indeed, the tails on $[T,+\ty)$, $T\gg1$ of the integrals  (\ref{E:eaz}) are small uniformly in $n$ (this comes from the boundedness
of $u_n$ in $\tilde \GG$), and on the finite interval $[0, T]$ the convergence is uniform. This implies that
$$
F(u_0)\le \inf_{u\in \overline {B_{ \tilde\FF}(0,M)} } F(u).
$$Since $u_0\in \overline {B_{ \tilde \FF}(0,M)} $,  we have (\ref{E:mh}).

\vspace{6pt}\textbf{Step 2.} To complete the proof, we need to show that $ F(u_0)=0.$    
\begin{lemma}\label{L:Min2} Under the conditions of Proposition \ref{P:P1},
the set
$$  U:=\{  \{\check u(\om_{m})\}: u\in \tilde \FF\}
$$
is dense in $ \ell^\ty_{10} $.
\end{lemma}
 Combining this with  the  Baire lemma, we get that for sufficiently large $M>0$
  $$  \tilde U:=\{  \{\check u(\om_{m})\}: u\in B_{ \tilde \FF} (0,M)\}
$$is dense in $B_{\ell^\ty_{10}}(0,1)$. Thus $ F(u_0)=0.$

 \end{proof}  \begin{proof}[Proof of Lemma \ref{L:Min2}]
 It is well known that the dual of $\ell^\ty_0$ is $\ell^1.$
Let us suppose that $h=\{h_m\}\in  \ell^1 $ is such that 
$$\lag h, \{\check u(\om_{m})\}\rag_{\ell^1, \ell^\ty_0}=0 $$
for all $u\in \tilde \FF$. Then  replacing in this equality $\check u (\om_{m}) $ by its integral representation, we get   
\begin{align*}& 0=\sum_{m=1 }^{+\ty}  \int^{+\ty}_0 e^{i\om_{m}s} u(s)\dd s\overline{h_{m}} =  \int^{+\ty}_0 u (s) \Big( \sum_{m =1 }^{+\ty} e^{i\om_{m}s} \overline{h_{m}} \Big) \dd s .
\end{align*}  Since  $\om_i\neq \om_j$ for $i\neq j$, by Lemma 3.10 in \cite{VN}, we have  $h_m=0$ for any $m\ge1$.  This proves that $U$ is dense.

 \end{proof}

\subsection{Application of the inverse mapping theorem}\label{S:1}

The proof is based on the inverse mapping theorem.  
 We project the system   onto the tangent space $T_{\tilde z}$ and  apply the inverse mapping theorem to the following    mapping
\begin{align}
\tilde\UU_\ty( \cdot) :   \FF&\ri T_{\tilde z}\cap \HH  , \nonumber \\
 u&\ri   P \UU_\ty(\tilde z,u) , \nonumber
\end{align} where $P $ is the orthogonal projection in $L^2$ onto $T_{\tilde z}$, i.e., $Pz= z-\Re\lag z, \tilde z\rag \tilde z, z\in L^2$. Notice that   $P ^{-1}:B_{T_{\tilde z}}(0,\de)\ri S$ is well defined for sufficiently small $\de>~0$.
The following result proves that $\tilde\UU_\ty$ is $C^1$.
\begin{proposition}\label{P:ppo}For a sufficiently small $\de>0$ the mapping
\begin{align}
 \UU_\ty(\tilde z, \cdot) :   B_\FF(0,\de)&\ri   \HH  , \nonumber \\
 u&\ri     \UU_\ty(\tilde z,u) , \nonumber
\end{align}
is $ C^1$. Moreover,   $\dd \UU_\ty(\tilde z, u)v=R_\ty(u,v)$, where 
\begin{align}\label{R:lim}
R_\ty(u,v):=\lim _{n\ri+\ty} R_{T_n}(u,v)\quad\text{in $\HH$,}
\end{align}and $R_t$ is the resolving operator of
\begin{align}
i\dot z &= -\Delta z+V(x)z  + u(t) Q(x)z+v(t) Q(x)\UU_t(\tilde z,u),\,\,\,\,
\label{E:hav6}\\
z\arrowvert_{\partial D}&=0,\label{E:ep6}\\
 z(0,x)&=z_0.\label{E:sp6}
\end{align}  

\end{proposition}

This proposition  implies  that $\tilde\UU_\ty\in C^1(B_\FF(0,\de))$.
By the definition of $T_n$, we have $\lim_{n\ri +\ty}\UU_{T_n}(\tilde z,0)=\tilde z $. Hence   $\UU_\ty(\tilde z,0)= \tilde z $ and  $\tilde\UU_\ty(0)=0$. We have $\dd \tilde\UU_\ty(0)v=R_\ty(0,v)$, which is invertible for $\tilde z\notin \EE_0$ in view of Theorem~\ref{T:Lin}. Thus applying the inverse mapping theorem, we complete the proof of Theorem~\ref{T:AnvL} for $\tilde z\notin \EE_0$.

In the case $\tilde z\in \EE_0$
the linearized system   is
not controllable, and $R_\ty$  is not invertible. Controllability near $\tilde z$ in finite time and for $d=1$
is proved by Beauchard and Coron \cite{BeCo}. They show that the linearized system
is controllable up to codimension one. This implies that the nonlinear system
is also controllable up to codimnsion one.   The controllability in the missed
directions is proved using the intermediate values theorem. In the case $d\ge1$ and $T=+\ty$, the proof repeats literally the arguments of \cite{BeCo}. 
We omit the details.

\begin{proof}[Proof of Proposition \ref{P:ppo}] See \cite{KBCL} for the proof the fact that $\UU_T(\tilde z, \cdot)$ is $C^1$ when $T$ is finite, $d=1$ and phase space is $H^3$.   
Let us show that  $\UU_\ty(\tilde z, \cdot)$ is differentiable at any $u\in B_\FF(0,\de)$ for sufficiently small $\de>0$. We need to prove that 
\begin{align}\label{E:mml}
\|\UU_\ty(\tilde z, u+v)-\UU_\ty(\tilde z, u )-R_\ty(u,v)\|_\HH=o(\|v\|_\FF).
\end{align}Notice that $h=\UU_t(\tilde z, u+v)-\UU_t(\tilde z, u )-R_t(u,v)$ is a solution of
\begin{align*}
i\dot h &= -\Delta h+V(x)h  + (u(t)+v(t)) Q(x)h+v(t) Q(x)R_t(u,v),\,\,\,\,
 \\
h\arrowvert_{\partial D}&=0, \\
 h(0,x)&=0. 
\end{align*}  Using Proposition \ref{L:lav} and Lemma \ref{L:LmpD}, we get
\begin{align*}
&\|h(\ty)\|_\HH \le C\int_0^{+\ty} (\|h(\tau)\|_{H^l}|u(\tau)+v(\tau)|+\|R_\tau(u,v)\|_{H^l}|v(\tau)|)\dd \tau\nonumber\\&\le C\int_0^{+\ty}  ( \|v\|_{C^m_0 }\| R_\cdot(u,v)\|_{W^{m,1} ([0,\tau  ],H^2_{(V)})} |u(\tau)+v(\tau)|e^{C(\|u+v\|_{C^m_0 }+1)\tau}\nonumber\\&\quad+\|R_\tau(u,v)\|_{H^l}|v(\tau)|)\dd \tau\nonumber\\&\le C\int_0^{+\ty}  ( \|v\|_{C^m_0 }^2\| \UU_\cdot(\tilde z,u)\|_{W^{m,1} ([0,\tau  ],H^2_{(V)})} |u(\tau)+v(\tau)|e^{C(\|u+v\|_{C^m_0 }+\|v\|_{C^m_0 }+2)\tau}\nonumber\\&\quad+\|v\|_{C^m_0 }\| \UU_\cdot(\tilde z,u)\|_{W^{m,1} ([0,\tau  ],H^2_{(V)})} |v(\tau)|e^{C (\|v\|_{C^m_0 }+1) \tau})\dd \tau \nonumber\\& \le C\|v\|^2_\FF,
\end{align*} for any $v\in B_\FF(0,\e)$, sufficiently small $\e>0$, and for sufficiently large $B>0$ in the definition of $\GG$. 

It remains to prove that $R_\ty(u,\cdot)$ is continuous in  $ B_\FF(0,\de)$. For $g:=R_t(u_1,v)-R_t(u_2,v)$ we have
\begin{align*}
i\dot g &= -\Delta g+V(x)g  + u_1(t)  Q(x)g+(u_1(t)-u_2(t))Q(x)R_t(u_2,v)\\&\quad+v(t) Q(x)(\UU_t(\tilde z ,u_1)-\UU_t(\tilde z,u_2)),\\
g\arrowvert_{\partial D}&=0, \\
 g(0,x)&=0. 
\end{align*} By Proposition \ref{L:lav}, 
\begin{align*} 
\|g(\ty)\|_\HH&\le C\int_0^{+\ty}(\|g(\tau)\|_{H^l}|u_1(\tau)|+\|R_\tau(u_2,v)\|_{H^l}|u_1(\tau)-u_2(\tau)|\nonumber\\&\quad+\|\UU_\tau(\tilde z ,u_1)-\UU_\tau(\tilde z,u_2))\|_{H^l}|v(\tau)|)\dd \tau=:I_1+I_2+I_3.
\end{align*}  Lemmas \ref{L:LmpD} and \ref{L:oe} imply
\begin{align*} I_1&\le C\int_0^{+\ty}  (\|R_\cdot(u_2,v)\|_{W^{m,1} ([0,\tau  ],H^2_{(V)})}\|u_1(\tau)-u_2(\tau)\|_{C_0^m}\nonumber \\&\quad+ \|\UU_\cdot(\tilde z ,u_1)-\UU_\cdot(\tilde z,u_2)))\|_{W^{m,1} ([0,\tau  ],H^2_{(V)})}||v(t)\|_{C_0^m})|u_1(\tau)|e^{C(\|u_1\|_{C_0^m}+1)\tau}\dd \tau\nonumber\\&\le C \|u_1-u_2\|_\FF.
\end{align*}  The terms $I_2,I_3$ are treated in a similar way. Thus we get   the continuity of $R_\ty(u,\cdot)$.

\end{proof}

\end{document}